\documentclass[preprint]{elsarticle}
\usepackage{amsmath,amsfonts,amsthm,url,color,amssymb}
\usepackage{slashbox}
\usepackage{algpseudocode, algorithm}
\newtheorem{theorem}{Theorem}[section]
\newtheorem{remark}[theorem]{Remark}
\newtheorem{proposition}[theorem]{Proposition}
\newtheorem{lemma}[theorem]{Lemma}
\newtheorem{define}[theorem]{Definition}
\newtheorem{corollary}[theorem]{Corollary}
\newcommand{\rmv}[1]{}
\newcommand{\Tr}{\mathrm{Tr}}
\newcommand{\fqn}{\mathbb{F}_{q^n}}
\newcommand{\F}{\mathbb{F}}
\newcommand{\fq}{\mathbb{F}_q}
\newcommand{\ord}{\mathrm{ord}}

\title{Existence of primitive $1$-normal elements in finite fields}
\author[Car]{L. Reis\fnref{fn1}}
\ead{lucasreismat@gmail.com}

\author[Car]{D. Thomson\corref{cor1}}
\ead{dthomson@math.carleton.ca}
\cortext[cor1]{Corresponding author}
\fntext[fn1]{Permanent address: 
Departamento de Matem\'{a}tica,
Universidade Federal de Minas Gerais,
UFMG,
Belo Horizonte MG (Brazil),
 30123-970.}

\address[Car]{School of Mathematics and Statistics, Carleton University, 1125 Colonel By Drive, Ottawa ON (Canada), K1S 5B6}

\begin{document}

\begin{abstract}
An element $\alpha \in \fqn$ is \emph{normal} if $\mathcal{B} = \{\alpha, \alpha^q, \ldots, \alpha^{q^{n-1}}\}$ forms a basis of $\fqn$ as a vector space over $\fq$; in this case, $\mathcal{B}$ is a \emph{normal basis} of $\fqn$ over $\fq$. The notion of $k$-normal elements was introduced in Huczynska et al (2013). Using the same notation as before, $\alpha$ is $k$-normal if $\mathcal{B}$ spans a co-dimension $k$ subspace of $\fqn$. It can be shown that $1$-normal elements always exist in $\fqn$, and Huczynska et al (2013) show that elements that are simultaneously primitive and $1$-normal exist for $q \geq 3$ and for large enough $n$ when $\gcd(n,q) = 1$ (we note that primitive $1$-normals cannot exist when $n=2$). In this paper, we complete this theorem and show that primitive, $1$-normal elements of $\fqn$ over $\fq$ exist for all prime powers $q$ and all integers $n \geq 3$, thus solving Problem 6.3 from Huczynska, et al (2013).
\end{abstract}


\maketitle

\section{Introduction}
Let $q$ be a power of a prime, there is a unique (up to isomorphism) finite finite of $q$ elements, denoted $\fq$. For all positive integers $n$, the finite extension field $\fqn$ of $\fq$ can be viewed as a vector space over $\fq$. Finite extension fields display cyclicity in multiple forms; for example, their Galois groups are cyclic of order $n$, generated by the Frobenius automorphism $\alpha_q(\alpha) = \alpha^q$ for any $\alpha \in \fqn$. The multiplicative group of $\fqn$, denoted $\fqn^*$ is a cyclic group of order $q^n-1$. 

An element $\alpha \in \fqn$ is \emph{primitive} if it is a generator of $\fqn^*$. 
An element $\alpha \in \fqn$ is \emph{normal} in $\fqn$ over $\fq$ if its Galois orbit is a spanning set for $\fqn$ as a vector space over $\fq$. Specifically, $\alpha$ is a normal element if $\mathcal{B} = \{\alpha, \alpha^q, \ldots, \alpha^{q^{n-1}}\}$ is a basis of $\fqn$ over $\fq$, whence we call $\mathcal{B}$ a \emph{normal basis} of $\fqn$ over $\fq$. The existence of normal elements of finite extension fields $\fqn$ over $\fq$ was established for all $q$ and $n$ by Hensel in 1888 and re-established by Ore in 1934 by studying properties of linearized polynomials. In this work, we draw on ideas extending from Ore. 

A natural question is to establish the existence of elements of $\fqn$ which are simultaneously primitive and normal over $\fq$. This was proven asymptotically when $q = p$ a prime by Davenport, and then asymptotically for all $q$ by Carlitz. The \emph{primitive normal basis theorem} was finally established for all $q,n$ by Lenstra and Schoof in 1988~\cite{lenstra} using a combination of character sums, sieving results and a computer search. Using more complicated sieving techniques, Cohen and Huczynska established the primitive normal basis theorem for all $q$ and $n$ without the use of a computer in 2003. 

Recently, Huczynska, et al~\cite{HMPT} defined \emph{$k$-normal elements} as generalizations of normal elements. In~\cite{HMPT}, they showed multiple equivalent definitions, we pick the most natural for this work. 

\begin{define}\label{def:k-normal}
Let $\alpha \in \fqn$, then $\alpha$ is $k$-normal over $\fq$ if its orbit under the cyclic (Frobenius) Galois action spans a co-dimension $k$ subspace of $\fqn$ over $\fq$; that is, if $V = \mathrm{Span}(\alpha, \alpha^q, \ldots, \alpha^{q^{n-1}})$, then $\dim_{\fq}(V) = n-k$. 
\end{define}

Under Definition~\ref{def:k-normal}, normal elements are $0$-normal elements, and all elements of $\fqn$ are $k$-normal for some $0 \leq k \leq n$. 
Notice that it is important to specify over which field subfield an element of $\fqn$ is $k$-normal. If not otherwise specified, when we say $\alpha \in \fqn$ is $k$-normal, we always assume it is $k$-normal over $\fq$. 

It can be shown (see Section~\ref{sec:background}) that there always exist $1$-normal elements of $\fqn$ over $\fq$. In~\cite{HMPT}, using a similar methodology to Lenstra-Schoof, the authors partially establish a primitive $1$-normal element theorem; that is, the existence of elements which simultaneously generate the multiplicative group of a finite fields and whose Galois orbit is a spanning set of a (Frobenius-invariant) hyperplane. 

\begin{theorem}[\cite{HMPT}, Theorem 5.10] \label{thm:mainpanario}
Let $q=p^e$ be a prime power and $n$ a positive integer not divisible by $p$. Assume that $n\ge 6$ if $q\ge 11$ and that $n\ge 3$ if $3\le q\le 9$. Then there exists a primitive $1$-normal element of $\F_{q^n}$ over $\F_q$.
\end{theorem}

This paper deals with completing Theorem~\ref{thm:mainpanario} by covering the following cases:
\begin{enumerate}
\item Including the cases $n=2,3,4,5$ for $q \geq 11$. 
\item Expanding Theorem~\ref{thm:mainpanario} to include $q=2$.
\item Removing the restriction where $p \nmid n$. 
\end{enumerate}

In Section~\ref{sec:background} we present some of the necessary background for this paper. In Section~\ref{sec:charactersum}, we establish a general character sum estimate for the number of elements in $\fqn$ that are simultaneously primitive, $g$-free for some polynomial $g \in \fq[x]$ (whose definition will be established later) and having trace $\beta$ for any $\beta$ laying in a subfield of $\fqn$ over $\fq$. Many technical estimates for this character sum are presented in the appendix. In Section~\ref{sec:coprime} we complete Theorem~\ref{thm:mainpanario} when $p \nmid n$; our result is usually explicit, but in some problematic cases we can only do so asymptotically. In Section~\ref{sec:pdivn} we similarly prove the existence of primitive $1$-normal elements when $p$ divides $n$. 

Our methods are largely analytic and rely on a number of estimations presented in Appendix~\ref{sec:estimations}. In all cases, we find effective bounds for when our analytic methods fail and in the remaining cases we find primitive $1$-normal elements by computer search. The pseudocode for the search is presented in Appendix~\ref{sec:pseudocode}.

We now give a succinct statement of the main result of this paper. 

\begin{theorem}\label{MainTheorem} \textup{\textbf{(The Primitive $1$-Normal Theorem)}}
Let $q$ be a prime power and let $n \geq 3$ be a positive integer. Then there exists a primitive $1$-normal element of $\fqn$ over $\fq$. Furthermore, when $n=2$ there is no primitive $1$-normal element of $\F_{q^2}$ over $\fq$. 
\end{theorem}

\section{Background material}\label{sec:background}

\subsection{Finite fields as Frobenius-modules}
We follow the description of finite fields as Frobenius modules, as in~\cite{lenstra}. Recall that the Frobenius $q$-automorphism $\sigma_q\colon \fqn\to\fqn$ is a $\fq$-linear map. For any $f(x) = \sum_{i=0}^{s} a_i x^i \in \fq[x]$, define
\[ f \circ \alpha = \sum_{i=0}^{s} a_i \sigma_q^i(\alpha) = \sum_{i=0}^s a_i \alpha^{q^i},\]
for any $\alpha$ in the algebraic closure of $\fq$. For any $\alpha \in \fqn$, we observe that $(x^n-1)\circ \alpha = 0$.

If $f(x) = \sum_{i=0}^{n-1} a_i x^i$, then define $F(x) = \sum_{i=0}^{n-1} a_i x^{q^i}$. The polynomial $F$ is the \emph{linearized $q$-associate of $f$} and $f$ is the \emph{conventional $q$-associate of $F$}. We outline some of the nice properties of $q$-associate polynomials. For the remainder of this section, lower-case polynomials represent conventional $q$-associates of their linearized $q$-associate with corresponding capital letters. 

\begin{proposition}\textup{\cite[Theorem 3.62]{LN}}\label{prop:linearized}
\begin{enumerate}
\item The following are equivalent:
\begin{enumerate}
\item $H = F(G_1)$, for some linearized polynomial $G_1 \in \fq[x]$.
\item $H = FG_2$, for some linearized polynomial $G_2 \in \fq[x]$.
\item $h = fg$, for some $g \in \fq[x]$.
\end{enumerate}
\item Let $h = f + g$, then $h \circ \alpha = f \circ \alpha + g \circ \alpha$.
\item Let $h = fg$, then $H = F(G)$; hence $(fg)\circ \alpha = F\circ(g\circ\alpha) = F(G(\alpha))$. 
\end{enumerate}
\end{proposition}

Since $(x^n-1)\circ\alpha = 0$ for all $\alpha \in \fqn$, we define the \emph{minimal polynomial of $\alpha$ with respect to $\sigma_q$}, denoted $m_{\sigma_q, \alpha}$, to be the monic polynomial of minimal degree for which $m_{\sigma_q, \alpha} \circ \alpha = 0$. We note that this is well defined, since if $m$ and $m'$ are two such minimal polynomials, then $(m - m')\circ \alpha = m\circ \alpha - m' \circ \alpha = 0$, but $m-m'$ has degree at most $n-1$, contradicting the minimality of $m$ and $m'$ unless $m = m'$. 

We obtain the following vital proposition immediately from Proposition~\ref{prop:linearized}. 

\begin{proposition}
For all $\alpha \in \fqn$, $m_{\sigma_q, \alpha}(x)$ divides $x^n-1$. 
\end{proposition}

Suppose $m_{\sigma_q, \alpha}$ has degree $n-k$ for $0 < k < n$, then $m_{\sigma_q, \alpha}$ is a non-trivial divisor of $x^n-1$; moreover, if $m_{\sigma_q, \alpha}(x) = \sum_{i=0}^{n-1} a_i x^i$, then $0 = m_{\sigma_q, \alpha} \circ \alpha = \sum_{i=0}^{n-1} a_i \alpha^{q^{i}}$ is a vanishing non-trivial linear combination of the Galois conjugates of $\alpha$; hence, $\alpha$ is not a normal element of $\fqn$ over $\fq$. Moreover, by the minimality of $m_{\sigma_q, \alpha}$, $\alpha$ is not annihilated by any non-trivial linear combination of $\{\alpha, \alpha^q, \ldots, \alpha^{q^{n-k-1}}\}$, hence $\alpha$ is $k$-normal in $\fqn$ over $\fq$. We summarize this below.

\begin{proposition}\label{prop:k-normaliff}
Any $\alpha \in \fqn$ is $k$-normal over $\fq$ if and only if $m_{\sigma_q, \alpha} \in \fq[x]$, the minimal polynomial of $\alpha$ under $\sigma_q$, has degree $n-k$. 
\end{proposition}

We observe that $x^n-1 = (x-1) (1 + x + \cdots + x^{n-1})$ for any field. Moreover, $(1 + x + \cdots + x^{n-1})\circ \alpha = \alpha + \alpha^q + \cdots + \alpha^{q^{n-1}} = \Tr_{q^n/q}(\alpha)$, where $\Tr_{q^n/q}$ is the trace map from $\fqn$ to $\fq$. 

\begin{corollary}\label{cor:extreme-k}
\begin{enumerate}
\item We have $\alpha \in \fq$ if and only if $m_{\sigma_q, \alpha}(x) = (x-1)$.
\item Let $\Tr_{q^n/q}(\alpha) = 0$, then $\alpha$ is not $0$-normal.
\end{enumerate}
\end{corollary}

Finally, we observe that every monic factor of $x^n-1$ of degree $n-k$ does indeed produce $k$-normal elements. In fact, the following theorem which also appears in~\cite[Corollary 3.71]{LN} can be shown by a counting argument. 

\begin{theorem}
Let $f$ be a monic divisor of $x^n-1$. Then there exists $\alpha \in \fqn$ for which $f = m_{\sigma_q, \alpha}$. 
\end{theorem}

We now focus on $1$-normal elements.

\begin{corollary}
Let $n \geq 2$, then there exists a $1$-normal element of $\fqn$ over $\fq$ for all $q$. 
\end{corollary}
\begin{proof} 
As noted previously, $x^n-1 = (x-1)(1 + x + \cdots + x^{n-1})$ for all $n \geq 2$. 
\end{proof}

We use Proposition~\ref{prop:k-normaliff} to give an exact count of the number of $k$-normal elements of $\fqn$ over $\fq$. First, we require the polynomial analogue of the Euler totient function.

\begin{define}
Let $f(x)$ be a monic polynomial with coefficients in $\F_q$. The Euler Phi Function for polynomials over $\F_q$ is given by $$\Phi_q(f)=\left |\left(\frac{\F_q[x]}{\langle f\rangle}\right)^{*}\right |,$$ where $\langle f\rangle$ is the ideal generated by $f(x)$ in $\F_q[x]$. 
\end{define}

Finally, we present an interesting formula for the number of $k$-normal elements over finite fields:

\begin{corollary}[\cite{HMPT}, Theorem 3.5] \label{count} The number $N_k$ of $k$-normal elements of $\F_{q^n}$ over $\F_q$ is given by
\begin{equation}\label{eq2} N_k = \sum_{h|x^n-1\atop{\deg(h)=n-k}}\Phi_q(h),\end{equation}
where the divisors are monic and polynomial division is over $\F_q$.
\end{corollary}

\subsection{Characters, characteristic functions and Gauss sums}\label{sec:charsumbground}
For the background on characteristic functions for our elements of interest, we follow the treatment in~\cite{HMPT}, which is essentially the same as in~\cite{cohen} and elsewhere. 

Let $\fqn^* = \langle \alpha \rangle$ and for a fixed integer $j$ with $0 \leq d < n$, define the map
\[ \eta_d(\alpha^k) = e^{\frac{2\pi i dk}{q^n-1}}, \quad k = 0, 1, \ldots, n-1.\]
The map $\eta_d\colon$ is a \emph{multiplicative character} of $\fqn$. Observe that $\eta_d(\alpha^s)\eta_d(\alpha^t) = \eta_d(\alpha^{s+t})$; hence, each $\eta_d$ is a homomorphism $\fqn^*\to \mathbb{C}^*$. Moreover, since $\alpha^{q^n-1} = 1$, we have $\eta_d(\alpha^k)^{q^n-1} = 1$; hence, $\eta_d$ maps to the unit circle in $\mathbb{C}$. Finally, since $\eta_d(\alpha^k) = \eta_1(\alpha^k)^d$ for all $d$; and the set of multiplicative characters of $\fqn$ form a cyclic group of order $q^n-1$ with identity element $\eta_0$. 

Let $p$ be the characteristic of $\fq$ and denote by $\chi_1$ the mapping $\chi_1 \colon \fqn \to \mathbb{C}$ by 
\[\chi_1(\alpha) = e^{\frac{2\pi i \Tr_{q/p}(\alpha)}{p}} \quad \alpha \in \fq.\]
The character $\chi_1$ is the \emph{canonical additive character} of $\fq$. Moreover, the mapping $\chi_\beta(\alpha) = \chi_1(\beta\alpha)$ for all $\alpha \in \fq$ is an additive character of $\fq$, and all additive characters of $\fq$ arise as $\chi_\beta$ for some $\beta$. 

Let $\fqn$ be an extension of $\fq$ with canonical additive character $\mu_1$. Then $\mu_1$ and $\chi_1$ are related by transitivity of the trace function, namely
\[ \mu_1(\alpha) = \chi_1(\Tr_{q^n/q}(\alpha)), \quad \alpha \in \fqn. \]
We call $\mu_1$ the lift of $\chi_1$ from $\fq$ to $\fqn$, and will make use of this fact later. 

\subsubsection{On freeness}
The notion of \emph{freeness} follows from the Vinogradov formulas, as follows. 

\begin{define}\label{def:free}
\begin{enumerate}
\item Let $d$ divide $q^n-1$, then $\alpha \in \fqn^*$ is $d$-free if $\alpha = \beta^m$ for any divisor $m$ of $d$ implies $d=1$. 
\item Let $m$ be a divisor of $x^n-1$, then $\alpha$ is $m$-free if $\alpha = h \circ \beta$ for any divisor $h$ of $m$ implies $h=1$. 
\end{enumerate}
\end{define}

We use a re-characterization of free elements that can be found in~\cite[Propositions 5.2, 5.3]{HMPT}, for example. 

\begin{proposition}\label{prop:free}
\begin{enumerate}
\item $\alpha \in \fqn^*$ is $d$-free if and only if $\gcd(d, \frac{q^n-1}{\ord(\alpha)}) = 1$, where $\ord(\alpha)$ is the multiplicative order of $\alpha$. 
\item $\alpha \in \fqn$ is $m$-free if and only if $\gcd(m, \frac{x^n-1}{m_{\sigma_q, \alpha}(x)}) = 1$. 
\end{enumerate}
\end{proposition}

Intuitively, we can interpret an element $\alpha$ as $\beta$-free (in either multiplicative or additive contexts) if its minimal annihilator in a sense contains all prime factors of $\beta$. We can use Proposition~\ref{prop:free} to show the utility of these characteristic functions. 

\begin{corollary}\label{cor:free}
\begin{enumerate}
\item Let $d = q^n-1$, then $\alpha \in \fqn$ is $d$-free if and only if $\alpha$ is primitive. 
\item Let $m = x^n-1$, then $\alpha \in \fqn$ is $m$-free if and only if $\alpha$ is normal. 
\item Let $\gcd(p,n) = 1$, let $m = \frac{x^n-1}{x-\zeta}$ for some $\zeta \in \fq$ and let $\alpha \in \fqn$ be $m$-free, then $\alpha$ is either $1$-normal or $0$-normal. 
\end{enumerate}
\end{corollary}
\begin{proof}
The first two assertions are obvious from Proposition~\ref{prop:free}, so we prove only the final assertion. The condition $\gcd(p,n)=1$ gives that $x^n-1$ has no repeated roots. Let $m_{\sigma_q, \alpha}$ be the minimal polynomial of $\alpha$, then $\gcd(m, \frac{x^n-1}{m_{\sigma_q, \alpha}}) = 1$ implies that either $\frac{x^n-1}{m_{\sigma_q, \alpha}} = 1$ or $x-\zeta$; hence $m_{\sigma_q, \alpha} = x^n-1$ or $m_{\sigma_q, \alpha} = m$, respectively.
\end{proof}

Let $t$ be a positive divisor of $q^n-1$, let $\theta(t) = \frac{\varphi{t}}{t}$ and for $w \in \fqn^*$ let 
\[ \omega_t(w) = \theta(t) \underbrace{\sum_{d|t} \frac{\mu(d)}{\varphi(d)} \sum_{(d)} \eta_{(d)}}_{\int_{d|t} \eta_{(d)}}(w),\]
where $\mu$ is the M\"obius function, $\varphi$ is the Euler totient function, $\chi_{(d)}$ is a typical multiplicative character of order $d$ and the inner sum runs over all multiplicative characters of order $d$.

Similarly, let $T$ be a monic divisor of $x^n-1$, let $\Theta(T) = \frac{\Phi(T)}{q^{\deg(t)}}$ and for $w \in \fqn$ let
\[ \Omega_T(w) = \Theta(T) \underbrace{\sum_{D|T} \frac{\mathcal{M}(d)}{\Phi(D)} \sum_{(D)} \chi_{(D)}}_{\int_{D|T} \chi_{(D)}}(w),\]
where $\mathcal{M}$ and $\Phi$ are the polynomial analogues of $\mu$ and $\varphi$ and $\chi_{(D)}$ is an additive character for which $D \circ \chi_{(D)} = \chi_0$ and $D$ is minimal (in terms of degree) with this property. Alternatively, we may write $\chi_{(D)}=\chi_{\delta_D}$, for some $\delta_D\in \F_{q^n}$ and $\chi_{\delta_D}(w):=\chi(\delta_D\cdot w)$. In particular, $D \circ \chi_{(D)} = \chi_0$ and $D$ is minimal with this property if and only if the minimal polynomial of $\delta_D$ with respect to $\sigma_q$ is $D(x)$. For instance $\chi_{x-1}$ runs through $\chi_{c}$ for $c\in \F_{q}^*$. Also, for $D=1$, $\chi_{(D)}(x)=\chi_0(x)$ is the trivial additive character.

\begin{theorem}\textup{\cite[Section 5.2]{HMPT}}\label{thm:charfree}
\begin{enumerate}
\item Let $w \in \fqn^*$ and let $t$ be a positive divisor of $q^n-1$, then 
\[\omega_t(w) = \theta(t) \int_{d|t} \eta_{(d)}(w) = \begin{cases} 1 & \text{if $w$ is $t$-free,} \\ 0 & \text{otherwise.} \end{cases}\]
\item Let $w \in \fqn$ and let $T$ be a monic divisor of $x^n-1$, then 
\[\Omega_T(w) = \Theta(T) \int_{D|T} \chi_{(D)}(w) = \begin{cases} 1 & \text{if $w$ is $T$-free,} \\ 0 & \text{otherwise.} \end{cases}\]
\end{enumerate}
\end{theorem}

From Theorem~\ref{thm:charfree}, if $t=q^n-1$ and $T = x^n-1$, then $\omega_t\Omega_T$ is the characteristic function for elements $w \in \fqn$ that are simultaneously primitive and normal over $\fq$. If instead $\gcd(p,n) = 1$ and $T = \frac{x^n-1}{x-1}$, then by Corollary~\ref{cor:free}, $\omega_t\Omega_T$ is a characteristic function for elements that are either primitive and normal, or primitive with minimal polynomial $m_{\sigma_q, w}(x) = \frac{x^n-1}{x-1}$; in the second case, these are a subset of the primitive $1$-normal elements. From Corollary~\ref{cor:extreme-k}, however, we know that if $\frac{x^n-1}{x-1} \circ w = \Tr_{q^n/q}(w) = 0$, then $w$ is not $0$-normal; hence, to prove existence of primitive $1$-normals it is enough to prove the existence of $q^n-1$-free, $\frac{x^n-1}{x-1}$-free, trace-$0$ elements. To apply character sum arguments, we require a further character sum for elements with prescribed trace. 

\subsubsection{Elements with prescribed trace}
In~\cite{HMPT}, the authors require only a characteristic function for trace-$0$ elements of $\fqn$ over $\fq$, but in this work we will need a more general characteristic function for elements whose subfield trace function takes any prescribed value. 

It is well known that $\fqn$ contains a subfield of order $m$ if and only if $m$ divides $n$. Let 
\[ T_{m,\beta}(w) = \begin{cases} 1 & \text{if $\Tr_{q^n/q^m}(w) = \beta$,} \\
                                       0 & \text{otherwise.}
				             \end{cases}
\]
We require a character sum for $T_{m,\beta}$. 

Let $\lambda$ be the canonical additive character of $\fq$ with lifts $\lambda_m$ and $\chi$ to $\F_{q^m}$ and $\fqn$, respectively; that is $\lambda_m(w) = \lambda(\Tr_{q^m/q}(w))$ and $\chi(w) = \lambda(\Tr_{q^n/q}(w)) = \lambda(\Tr_{q^m/q}(\Tr_{q^n/q^m}(w))) = \lambda_m(\Tr_{q^n/q^m}(w))$, by transitivity of trace. 

Observe that $T_{m,\beta}$ can be written
\[ T_{m,\beta}(w) = \frac{1}{q^m} \sum_{d \in \F_{q^m}} \lambda_m(d (\Tr_{q^n/q^m}(w) - \beta)) = \frac{1}{q^m} \sum_{d \in \F_{q^m}} \lambda_m(d \Tr_{q^n/q^m}(w-\alpha)),\]
where $\alpha \in \fqn$ is any element such that $\Tr_{q^n/q^m}(\alpha) = \beta$, 
since $\Tr_{q^n/q^m}(w) = \beta$ implies $\lambda_m(0) = 1$ is counted $q^m$ times, otherwise $\Tr_{q^n/q^m}(w-\alpha) = \beta'$ for some $\beta' \in \F_{q^m}^*$ and $\sum_{d \in \F_{q^m}} \lambda_m(d\beta') = 0$. 
Hence $$T_{m,\beta}(w) = \frac{1}{q^m} \sum_{d \in \F_{q^m}}\chi_d(w-\alpha)=\frac{1}{q^m} \sum_{d \in \F_{q^m}}\chi_d(w)\chi_d(\alpha)^{-1}.$$

\subsubsection{Gauss sums}
For any multiplicative character $\eta$ and additive character $\chi$, denote by $G(\eta, \chi)$ the \emph{Gauss sum} 
\[ G(\eta, \chi) = \sum_{w \in \fqn^*} \eta(w) \chi(w).\]

Certainly, if $\chi = \chi_1$ and $\eta = \eta_0$, then $G(\eta, \chi) = q-1$. However, it is well known that $G(\eta, \chi)$ is usually much smaller.

\begin{theorem}\textup{\cite[Theorem 5.11]{LN}}
Let $G(\eta, \chi) = \sum_{w \in \fqn^*} \eta(w) \chi(w)$. Then,
\[ G(\eta, \chi) 
   = \begin{cases} q^n-1 & \text{for $\eta = \eta_0$ and $\chi = \chi_0$,}\\
	                  -1 & \text{for $\eta = \eta_0$ and $\chi \neq \chi_0$,}\\
										 0 & \text{for $\eta \neq \eta_0$ and $\chi = \chi_0$.}
		 \end{cases}\]
If both $\eta \neq \eta_0$ and $\chi \neq \chi_0$, then $|G(\eta, \chi)| = q^{n/2}$. 
\end{theorem}

Throughout this paper we will always take Gauss sums $\sum_{w \in \fqn} \eta(w) \chi(w)$ including the element $0$ and, for this reason, we extend the multiplicative characters to $0$: we set $\eta_0(0)=1$ and $\eta(0)=0$ for any nontrivial multiplicative character $\eta$. Therefore, our Gauss sums take the following values:
\[ G(\eta, \chi):=  \sum_{w \in \fqn} \eta(w) \chi(w)
   = \begin{cases} q^n& \text{for $\eta = \eta_0$ and $\chi = \chi_0$,}\\
	                  0 & \text{for $\eta = \eta_0$ and $\chi \neq \chi_0$,}\\
										 0 & \text{for $\eta \neq \eta_0$ and $\chi = \chi_0$.}
		 \end{cases}\]
As before, if both $\eta \neq \eta_0$ and $\chi \neq \chi_0$, then $|G(\eta, \chi)| = q^{n/2}$.

\section{Existence of primitive, free elements with prescribed trace}\label{sec:charactersum}
In this section, we generalize the character sum argument for proving the existence of primitive $1$-normal elements from~\cite{HMPT} by providing a character sum expression for primitive $1$-normal elements with prescribed trace in a subfield. We use the expressions from Section~\ref{sec:charsumbground}, along with the estimations found in Appendix~\ref{sec:estimations} to prove our main results.

We highlight the following notation, which will become ubiquitous throughout the rest of this document. 

\begin{define}
Let 
\begin{equation}\label{def:W(t)}
W(t) = \begin{cases} \text{the number of squarefree divisors of $t$} & \text{if $t\in \mathbb Z$,} \\
                     \text{the number of monic divisors of $t$} & \text{if $t \in \fq[x]$ is monic.}
	     \end{cases}
\end{equation}
Moreover, let $w(t)$ be the number of distinct prime divisors (or monic irreducible divisors) of $t$ so that $W(t) = 2^{w(t)}$. 
Finally, for $t \in \mathbb{Z}$ let $d(t)$ be the number of divisors of $t$ so that $d(t) \geq W(t)$.
\end{define}

\begin{proposition}\label{propmain}
Let $q$ be a power of a prime $p$ and let $\F_{q^n}\supsetneq \F_{q^m}\supseteq \F_{q}$, where either $m=1$ or $m$ is a power of $p$. Also, let $f(x)$ be a polynomial not divisible by $x-1$ such that $f(x)$ divides $x^n-1$. Let $N$ be the number of elements $w\in \F_{q^n}$ such that $w$ is primitive, $f$-free over $\F_q$ and $\Tr_{q^n/q^m}(w)=\beta$, where $\beta\in \F_{q^m}$. If $\alpha$ is any element of $\F_{q^n}$ such that $\Tr_{q^n/q^m}(\alpha)=\beta$ and $a_c=\chi_c(\alpha)^{-1}$, then the following holds

\begin{align}
\frac{N}{\theta(q^n-1)\Theta(f)}=\frac{1}{q^m} &\left(q^n + \sum_{c\in \F_{q^m}}\right. a_c\int\limits_{d|q^n-1\atop{d\ne 1}} \int\limits_{D|f\atop {D\ne 1}}G(\eta_d, \chi_{\delta_D+c}) \nonumber \\ 
&+ \left.\sum_{c\in \F^*_{q^m}}a_c\int\limits_{d|q^n-1\atop{d\ne 1}} G(\eta_d, \chi_{c})\right).\label{eqn:id}
\end{align}

In particular, we have the following inequality:

\begin{equation}\label{eqn:main1}\frac{N}{\theta(q^n-1)\Theta(f)}> q^{n-m}-q^{n/2}W(q^n-1)W(f).\end{equation}
\end{proposition}

\begin{proof}
Combining the characteristic functions for primitivity, $f$-freeness and prescribed trace we obtain the following:

$$N=\frac{\theta(q^n-1)\Theta(f)}{q^m}\sum_{c\in \F_{q^m}}\sum_{w\in \F_{q^n}}\int\limits_{d|q^n-1}\displaystyle\int\limits_{D|f}\chi_c(\alpha)^{-1}\chi_{c}(w)\chi_{\delta_D}(w)\eta_d(w),$$
hence 
$$\frac{N}{\theta(q^n-1)\Theta(f)}=\frac{1}{q^m}\sum_{c\in \F_{q^m}}a_c\int\limits_{d|q^n-1}\displaystyle\int\limits_{D|f}G(\eta_d, \chi_{\delta_D+c}).$$

Now observe that, for each $D$ dividing $f(x)$, $D$ is not divisible by $x-1$. Since $m=1$ or $m$ is a power of $p$, $\delta_D\in \F_{q^m}$ if and only if $D$ divides $x^m-1=(x-1)^m$.  In particular, $\delta_D\not\in  \F_{q^m}$ unless $D=1$ and thus $\delta_D=0$. Therefore $\delta_D+c\ne 0$ unless $D=1$ and $c=0$. Our Gauss sums then take the following values:

\[
  G(\eta_d, \chi_{\delta_D+c})=\begin{cases}
      q^n & \text{if $(D,d,c) = (1,1,0)$} \\ 
			  0 & \text{if $d=1$ and $(D,c) \neq (1,0)$ or if $d\ne 1$ and $(D,c) \neq (1,0)$.}
    \end{cases}
\] 
Otherwise, $|G(\eta_d, \chi_{\delta_D+c})|=q^{n/2}$.

Using the identities above we obtain Equation~\eqref{eqn:id}. We know that, in the remaining Gauss sums in Equation~\eqref{eqn:id}, we have $|G(\eta, \chi)|=q^{n/2}$ and clearly $|a_c|=1$. Applying these estimates to Equation~\eqref{eqn:id} we obtain
\begin{align*}
\frac{N}{\theta(q^n-1)\Theta(f)}&\ge q^{n-m}-\frac{q^{n/2}}{q^m}\left(q^m(W(q^n-1)-1)(W(f)-1)\right.\\ & \quad +\left.(q^m-1)(W(q^n-1)-1)\right)\\
&> q^{n-m}-q^{n/2}W(q^n-1)W(f).
\end{align*}\qedhere
\end{proof}

In this paper, we focus on the case of primitive $1$-normal elements. For the particular case of $1$-normal elements, if $n\ge 2$, the number of $1$-normal elements of $\F_{q^n}$ over $\F_q$ is at least equal to $\Phi_q(T)$, where $T=\frac{x^n-1}{x-1}$.

\section{A completion of Theorem \ref{thm:mainpanario} for the case $\gcd(n, p)=1$}\label{sec:coprime}
From~\cite[Corollary 5.8]{HMPT}, we can easily deduce the following:

\begin{lemma}\label{sievepanario}
Suppose that $q$ is a power of a prime $p$, $n\ge 2$ is a positive integer not divisible by $p$ and $T(x)=\frac{x^n-1}{x-1}$. If the pair $(q, n)$ satisfies
\begin{equation}\label{sieve}W(T)\cdot W(q^n-1)< q^{n/2-1}, \end{equation}
then there exist $1$-normal elements of $\F_{q^n}$ over $\F_q$.
\end{lemma}

Inequality \eqref{sieve} is an essential step in the proof of Theorem \ref{thm:mainpanario} and it was first studied in \cite{cohen2}. Under the condition that $n\ge 6$ for $q\ge 11$ and $n\ge 3$ for $3\le q\le 9$, this inequality is not true only for a finite number of pairs $(q, n)$. Namely, we have the following.

\begin{theorem}\textup{\cite[Theorem 4.5]{cohen2}}\label{cohenpairs}
Let $q$ and $n$ be coprime, and assume that $n\ge 6$ if $q\ge 11$, and that $n\ge 3$ if $3\le q\le 6$. The set of $34$ pairs $(q, n)$ that do not satisfy Inequality~\eqref{sieve} is 
\begin{gather*}(4, 15), (13, 12), (7, 12), (11, 10), (4, 9), (9, 8), (5, 8), (3, 8), (8, 7), (121, 6), (61, 6),
\\
(49, 6), (43, 6), (37, 6), (31, 6), (29, 6), (25, 6), (19, 6), (13, 6), (11, 6), (7, 6), (5, 6),
\\ 
(9, 5) (4, 5), (3, 5), (9, 4), (7, 4), (5, 4), (3, 4), (8, 3), (7, 3), (5, 3), (4, 3).\end{gather*}
\end{theorem}


\subsection{The case $q=2$}
\begin{lemma}\label{aux3}
Suppose that $n\ne 15$ is odd, $q=2$ and $T(x)=\frac{x^n-1}{x-1}\in \F_2[x]$. For $n>9$, Inequality \eqref{sieve} holds.
\end{lemma}
\begin{proof}
Notice that $\frac{n+9}{5}+\frac{n}{7}+2< \frac{n}{2}-1$ for $n\ge 31$. From Proposition \ref{aux} and Lemma \ref{Cohen}, it follows that Inequality \eqref{sieve} holds for odd $n\ge 31$. The remaining cases can be verified directly. 
\end{proof}

\begin{theorem}
Suppose that $n\ge 3$ is odd. Then there exist a primitive $1$-normal element of $\F_{2^n}$ over $\F_2$.
\end{theorem}

\begin{proof}
According to Lemmas \ref{sievepanario} and \ref{aux3}, this statement is true for $n> 9$  if $n\ne 15$. For the remaining cases $n=3, 5, 7, 9$ and $15$ we use the following argument. Let $P$ be the number of primitive elements of $\F_{2^n}$ and $N_1$ the number of $1$-normal elements of $\F_{2^n}$ over $\F_2$; if $P+N_1> 2^n$, there exists a primitive $1$-normal element of $\F_{2^n}$ over $\F_2$. Notice that $P=\varphi(2^n-1)$ and, according to Lemma \ref{count}, $N_1\ge \Phi_2\left(\frac{x^n-1}{x-1}\right)$. By a direct calculation we see that 
$$\varphi(2^n-1)+\Phi_2\left (\frac{x^n-1}{x-1}\right)>2^n,$$
for $n=3, 5, 7, 9$ and $15$. This completes the proof.
\end{proof}

\subsection{The case $n \leq 5$.}
Here we give direct proofs of existence or non-existence of primitive $1$-normal elements in $\fqn$ over $\fq$ for small degrees $n\leq 5$. We also obtain some asymptotic existence results. First we show that primitive $1$-normal elements cannot exist when $n=2$, which also appears in~\cite[Section 6]{HMPT}. We present the short proof here for completeness.
\begin{lemma}\label{lem:nonexistn-1}
Let $n\geq2$ be an integer. Then there does not exist a primitive $(n-1)$-normal element in $\fqn$ over $\fq$. 
\end{lemma}
\begin{proof}
Suppose $\alpha \in \fqn$ is $(n-1)$-normal over $\fq$, then $\alpha^q = k\alpha$ for some $k \in \fq$. Hence, the order of $\alpha$ divides $(q-1)^2$, which is impossible for $n \geq 2$. 
\end{proof}

\begin{corollary}\label{prop:n=2}
There do not exist primitive $1$-normals in $\F_{q^2}$ over $\fq$. 
\end{corollary}

Now we present a large family of degrees $n$ for which there must exist primitive $1$-normal elements of $\fqn$ over $\fq$. We observe that this result also appears in~\cite{Alizadeh}, and we leave the short proof here for completeness. 

\begin{lemma}\label{lem:qprimmodn}
Let $q$ be primitive modulo a prime $n>2$, then there exists a primitive $1$-normal element in $\fqn$ over $\fq$. 
\end{lemma}
\begin{proof}
Since $q$ is primitive modulo $n$ and $n$ is prime, $x^n-1 = (x-1)(x^{n-1} + \cdots + x + 1)$ is the complete factorization of $x^n-1$ into irreducibles. The only possibilities for $k$-normality are $k=0,1,n-1,n$. Since $(n-1)$-normal elements are not primitive, we have all primitive trace-$0$ elements in $\fqn$ are primitive $1$-normal. The existence of such elements is guaranteed by \cite{Cohentrace}; see also Lemma~\ref{lem:cohen}, below. 
\end{proof}


\begin{corollary}\label{cor:small-n}
\begin{enumerate}
\item There exists a primitive $1$-normal element of $\F_{q^3}$ over $\fq$. 
\item If $q \equiv 3\pmod{4}$, then there exists a primitive $1$-normal element of $\F_{q^4}$ over $\fq$.
\item If $q \equiv 2,3\pmod{5}$, then there exists a primitive $1$-normal element of $\F_{q^5}$ over $\fq$. 
\end{enumerate}
\end{corollary}
\begin{proof}
Let $n=3$. By Lemma~\ref{lem:nonexistn-1}, there do not exist primitive $2$-normal elements in $\F_{q^3}$. Hence all primitive, trace-zero elements are $1$-normal. If $q\ne 4$, the existence of primitive and trace-zero elements is guaranteed by \cite{Cohentrace}. For the case $q=4$ we can directly find a primitive $1$-normal element in $\F_{4^3}$.

Let $n=4$, then $x^n-1 = (x-1)(x+1)(x^2+1)$ in any field. 
If $q \equiv 3\pmod{4}$, then $x^2+1$ is irreducible; hence $2$-normal elements must be annihilated by $x^2-1$ or by $x^2+1$. In the first case, $\alpha$ lies in the subfield $\F_{q^2}$, hence cannot be primitive. In the second case, $(x^2+1)\circ \alpha = \alpha^{q^2} + \alpha = 0$ implies $\alpha^{q^2-1} = -1$, and the order of $\alpha$ divides $2(q^2-1) < q^4-1$.  By Lemma~\ref{lem:nonexistn-1}, there do not exist primitive $3$-normal elements in $\F_{q^4}$. Therefore all primitive, trace-zero elements are $1$-normal and, again, the existence of such elements is guaranteed by \cite{Cohentrace}.

Let $n=5$.
If $q\equiv 2,3\pmod{5}$, then $q$ is primitive$\pmod{5}$, and existence is given by Lemma~\ref{lem:qprimmodn}.
\end{proof}

\begin{remark}
We further comment on a statement in~\cite[Section 6]{HMPT}, where the authors state without proof that if $q\equiv 1\pmod{4}$, then there do not exist primitive $(n-2)$-normal elements. We show here that this is not true. Let $f(x) = x^5-x-2$, a primitive polynomial over $\F_5$, and let $f(\alpha) = 0$. It is straightforward to check that $m_{\sigma_q, \alpha}(x) = (x-1)^2$, hence $\alpha$ is $3$-normal.
\end{remark}

The remaining cases after applying the direct methods of Lemmas~\ref{lem:nonexistn-1} and \ref{lem:qprimmodn} are:
\begin{enumerate}
\item $n=4$, $q \equiv 1\pmod{4}$, 
\item $n=5$, $q \equiv \pm 1\pmod{5}$.
\end{enumerate}
For these remaining cases, we proceed as follows: we try to obtain reasonable large values of $q$ for which Inequality~\eqref{sieve} holds. For the remaining cases, we verify directly the same inequality. Some exceptions arise and then we  directly verify the existence of a primitive $1$-normal element by computer search. We start with the following proposition.

\begin{proposition}\label{prop:n=4andn=5}
Let $q$ be a power of a prime $p$. Then for $n=4$ and $q>5.24\cdot 10^7$ odd or $n=5$ and $q\ge 2^{17}$, Inequality \eqref{sieve} holds for the pair $(q, n)$.
\end{proposition}

\begin{proof}
Recall that $T=\frac{x^n-1}{x-1}$. We first consider the case $n=4$. In particular, $W(T)=8$ and $q^4-1$ is divisible by $16$ and then $8\cdot W(q^4-1)\le d(q^4-1)$; that is, 
\begin{equation}\label{eqn:verifysieven=4}
W(T)W(q^4-1)\le d(q^4-1).
\end{equation}
But, according to Lemma \ref{divisor}, 
$$d(q^4-1)\le q^{\frac{4.264}{\log(\log(q^4-1))}}<q=q^{4/2-1},$$
since $\log(\log(q^4-1))>4.264$ for $q>5.24\cdot 10^7$. This shows that the pair $(q, n)$ satisfies Inequality \eqref{sieve}. 

For $n=5$ and $q>2^{17}$, note that $W(T)\le 16$, hence $$W(T)W(q^5-1)\le 16\cdot W(q^5-1).$$ According to Lemma \ref{estimation5}, we have the bound 
\begin{equation}
16\cdot W(q^5-1)<q^{1.5}=q^{5/2-1},\label{eqn:verifysieven=5}
\end{equation}
and then $(q, n)$ satisfies Inequality \eqref{sieve}.
\end{proof}

For the case $n=4$, $q\equiv 1\pmod 4$ and $q\le 5.24\cdot 10^7$ we verify directly that Inequality~\eqref{sieve} holds for all but $138$ values of $q$, the greatest being $q=21,013$. 
Using a \emph{Sage} program~\cite{sagemath}, we explicitly find such primitive $1$-normal elements for these values of $q$.

For $n=5$ and $q\equiv\pm 1\pmod 5$, we directly verify that $16\cdot W(q^5-1)\le q^{1.5}$ for all $71 < q < 2^{17}$. For $q \leq 71$, this inequality does not hold for most $q$. For $2 \leq q \leq 71$, we compute $W(T)$ (which is $4$ or $16$ according to $q$ of the form $5k-1$ or $5k+1$, repectively) and verify that Inequality~\eqref{sieve} holds with the exception of the cases $q=4, 9, 11, 16, 31, 61, 71$. For these cases, using a \emph{Sage} program~\cite{sagemath}, we explicitly find primitive $1$-normal elements.

The pseudocode for the explicit search for a primitive $1$-normal element can be found in Appendix~\ref{sec:pseudocode}.

Combining Proposition~\ref{prop:n=4andn=5} and these remarks, we obtain the following.

\begin{corollary}
Suppose $q$ is a power of a prime $p$ and $n=3, 4, 5$ with $\gcd(p, n)=1$. Then there exists a primitive $1$-normal element of $\F_{q^n}$.
\end{corollary}

\section{Existence results for primitive $1$-normality when $p$ divides $n$}\label{sec:pdivn}
In this section, we prove the existence of primitive $1$-normal elements using a variety of methods. In Section~\ref{sec:n=p^2s}, we use an existence result about primitive elements with prescribed trace in order to prove the existence of primitive $1$-normal elements of $\fqn$ over $\fq$ when $n=p^2s$, where $p$ is the characteristic of $\fq$. In Section~\ref{sec:n=ps}, we use the estimations found in Appendix~\ref{sec:estimations} to prove the existence of primitive $1$-normal elements of $\fqn$ over $\fq$ for $n=ps$ with $\gcd(p,s)=1$.

\subsection{The case $n=p^2s$}\label{sec:n=p^2s}
In this section we prove the existence of primitive $1$-normal elements in $\fqn$ over $\fq$ whenever $p^2$ divides $n$, where $p$ is the characteristic of $\fq$. We begin with a crucial lemma from~\cite{Cohentrace}.

\begin{lemma}\label{lem:cohen}\textup{\cite{Cohentrace}}
If $n \geq 3$ and $(q,n) \neq (4,3)$, then for every $a \in \mathbb{F}_{q}$, there exists a primitive element $\alpha \in \mathbb{F}_{q^n}$ such that $\mathrm{Tr}_{q^n/q}(\alpha) = a$. Moreover, if $n=2$ or $(q,n) = (4,3)$, then, for every nonzero $a \in \mathbb{F}_{q}$, there exists a primitive element $\alpha \in \mathbb{F}_{q^n}$ such that $\mathrm{Tr}_{q^n/q}(\alpha) = a$.
\end{lemma}

We now give a general result which characterizes $1$-normal elements in $\fqn$ over $\fq$ based on their corresponding $1$-normal projections into the subfield $\F_{q^{ps}}$. 

\begin{lemma}\label{lem:n=p^2s}
Suppose that $\fq$ has characteristic $p$ and let $n=p^2s$ for any $s \geq 1$. Let $x-\zeta$ be a divisor of $x^n-1$, then $\alpha \in \fqn$ has minimal polynomial $m_{\sigma_q, \alpha} = \frac{x^n-1}{x-\zeta}$ if and only if $\beta = \Tr_{q^n/q^{ps}}(\alpha)$ has minimal polynomial $m_{\sigma_q, \beta}(x) = \frac{x^{ps}-1}{x-\zeta}$. 
\end{lemma}
\begin{proof}
We observe that if $x-\zeta$ is a divisor of $x^n-1$, then $x-\zeta$ is a factor of $x^s-1$ and $x^{ps}-1$ as well. Suppose $m_{\sigma_q, \alpha}(x) = \frac{x^n-1}{x-\zeta} = \frac{x^n-1}{x^{ps}-1}\cdot \frac{x^{ps}-1}{x-\zeta}$. Then $0 = m_{\sigma_q, \alpha}\circ\alpha = \frac{x^n-1}{x^{ps}-1} \circ \beta$; hence, $m_{\sigma_q,\beta}$ divides $\frac{x^{ps}-1}{x-\zeta}$. Moreover, $m_{\sigma_q, \alpha}(x)$ clearly divides $m_{\sigma_q, \beta}(x)\frac{x^n-1}{x^{ps}-1}$, so $\frac{x^{ps}-1}{x-\zeta}$ divides $m_{\sigma_q, \beta}$. 

Conversely, suppose that $m_{\sigma_q, \beta}(x) = \frac{x^{ps}-1}{x-\zeta}$. Then $m_{\sigma_q,\alpha}(x)$ divides 
$$\frac{x^n-1}{x^{ps}-1}\cdot \frac{x^{ps}-1}{x-\zeta} = \frac{x^n-1}{x-\zeta}.$$
Suppose that $m_{\sigma_q, \alpha}$ strictly divides $\frac{x^n-1}{x-\zeta}$, then there exists some irreducible $f(x)$ dividing $x^n-1$ such that $m_{\sigma_q,\alpha}$ divides $\frac{x^n-1}{f(x)(x-\zeta)}$. However, $x^n-1 = (x^s-1)^{p^2}$, hence $f$ irreducible and $x-\zeta$ a repeated root of $x^n-1$ implies that $f(x)$ divides $x^s-1$ and $f(x)(x-\zeta)$ divides $x^{ps}-1 = (x^s-1)^p$. Thus,
\[ 
0 = \frac{x^n-1}{f(x)(x-\zeta)}\circ \alpha = \frac{x^{ps}-1}{f(x)(x-\zeta)}\cdot \frac{x^n-1}{x^{ps}-1}\circ \alpha = \frac{x^{ps}-1}{f(x)(x-\zeta)}\circ \beta,\]
contradicting $m_{\sigma_q,\beta}(x) = \frac{x^{ps}-1}{x-\zeta}$. 
\end{proof}

In Lemma~\ref{lem:n=p^2s} we are, of course, primarily concerned with the case $\zeta = 1$, since $x-1$ is a divisor of $x^n-1$ over any field. We now combine Lemmas~\ref{lem:cohen} and~\ref{lem:n=p^2s} to yield the main result of this section. 

\begin{proposition}\label{prop:n=p^2s}
Suppose that $\fq$ has characteristic $p$ and that $n$ is a positive integer divisible by $p^2$. Then there exists a primitive $1$-normal element of $\fqn$ over $\fq$.
\end{proposition}
\begin{proof}
Let $n= p^2s = p\cdot ps$. Let $\beta$ be an element with minimal polynomial $\frac{x^{ps}-1}{x-1}\ne 1$. In particular, such a $\beta$ is nonzero. Since $p = n/ps \geq 2$, by Lemma~\ref{lem:cohen}, there is a primitive element $\alpha \in \fqn$ with $\Tr_{q^n/q^{ps}}(\alpha)=\beta$. Thus, from Lemma~\ref{lem:n=p^2s}, such an $\alpha$ is primitive $1$-normal.
\end{proof}

\subsection{The case $n=ps$ with $\gcd(s,p)=1$}\label{sec:n=ps}
Let $n=ps$ with $\gcd(p,s) = 1$. Then $x^n-1 = (x^s-1)^p$. Moreover, observe that $\left(\frac{x^s-1}{x-1}\right)^p \circ \alpha = \frac{x^n-1}{x^p-1} = \Tr_{q^n/q^p}(\alpha)$. 

\begin{proposition}\label{characterization}
Let $T(x) = \frac{x^s-1}{x-1}$; then $\alpha \in \fqn$ has minimal polynomial $\frac{x^n-1}{x-1}$ if and only if it is $T$-free and $\Tr_{q^n/q^p}(\alpha)$ has minimal polynomial $\frac{x^p-1}{x-1}$. 
\end{proposition}
\begin{proof}
Suppose $\alpha$ is $T$-free as in the hypothesis and $\beta = \Tr_{q^n/q^p}(\alpha)$ has minimal polynomial $m_{\sigma, \beta}(x) = \frac{x^p-1}{x-1}$. Then,
\[ \frac{x^n-1}{x-1}\circ \alpha = \frac{x^{ps}-1}{x-1}\circ \alpha = \Tr_{q^n/q}(\alpha) = \Tr_{q^p/q}(\Tr_{q^n/q^p}(\alpha)) = \Tr_{q^p/q}(\beta) = 0,\]
since $\beta$ is annihilated by $\frac{x^p-1}{x-1}$. Since $\alpha$ is $T$-free, $\gcd(T, \frac{x^n-1}{m_{\sigma,\alpha}(x)}) = 1$; that is, $m_{\sigma, \alpha}(x) = (x-1)^d \frac{x^n-1}{x^p-1}$ for some $0 \leq d \leq p$. The minimality of $\beta$ gives that $d = p-1$; hence $m_{\sigma,\alpha}(x) = \frac{x^n-1}{x-1}$. 

The reverse assertion is straightforward from the minimality of $\frac{x^n-1}{x-1}$. 
\end{proof}

The above proposition allows us to count a certain subset of $1$-normal elements. We can use similar methods as Lenstra-Schoof to show existence. We start with a version of Lemma \ref{sievepanario} in the case when $p$ divides $n$.

\begin{lemma}\label{sievep}
Suppose that $q$ is a power of a prime $p$ and $n=ps$, where $s\ge 2$ is a positive integer not divisible by $p$. Set $T(x)=\frac{x^s-1}{x-1}$. If the pair $(q, s)$ satisfies
\begin{equation}\label{s>2}W(T)\cdot W(q^{ps}-1)\le q^{p(\frac{s}{2}-1)}, \end{equation}
then there exist $1$-normal elements of $\F_{q^n}$ over $\F_q$.
\end{lemma}

\begin{proof}
Combining Propositions~\ref{propmain} and~\ref{characterization} with $n=ps$, $m=p$ and $f(x)=T(x)=\frac{x^s-1}{x-1}$, if there is no primitive $1$-normal element in $\F_{q^n}$, then
$$0=\frac{N}{\theta(q^n-1)\Theta(T)}>q^{p(s-1)}-q^{\frac{ps}{2}}W(q^{ps}-1)W(T),$$
and hence
$$ q^{p(\frac{s}{2}-1)}< W(q^{ps}-1)W(T).$$
This is a contradiction with our assumptions.
\end{proof}

We note that this criterion does not include the case $s=1$ and, for $s=2$, is worthless. Using different combinatorial arguments we will see how to obtain the desired result for $s=1, 2$. First, we deal with the cases $s\ge 3$.

\subsubsection{Cases $n=ps$ with $3\le s\le 5$.}

\begin{proposition}
Suppose that $q$ is a power of a prime $p$. For $n=ps$, where either $s=3$ and $p\ge 29$ or $s=4,5$ and $p\ge 11$, the pair $(q, s)$ satisfies Inequality \eqref{s>2}.
\end{proposition}

\begin{proof}
Notice that $T(x)=\frac{x^s-1}{x-1}$ has degree $s-1$, hence $W(T)\le 2^{s-1} = \frac{q^{s\log_q 2}}{2}$.
According to Proposition \ref{mainestimation}, for $p\ge 11$ we have
$$W\left(\frac{q^{ps}-1}{q^{s}-1}\right)< q^{\frac{s(p-1)}{2+\log_2 p}}.$$
We have the trivial bound $W(q^s-1)\le d(q^s-1)\le 2q^{\frac{s}{2}}$ and then
$$W(q^{ps}-1)W(T)<q^{\frac{s(p-1)}{2+\log_2 p}+\frac{s}{2}+s\log_q 2}\le q^{\frac{s(p-1)}{2+\log_2 p}+\frac{s}{2}+s\log_p 2},$$
since $q\ge p$.
Therefore, if
$$p\left(\frac{s}{2}-1\right)\ge \frac{s(p-1)}{2+\log_2 p}+\frac{s}{2}+s\cdot \frac{\log 2}{\log p},$$
then $(q, s)$ satisfies Inequality \eqref{s>2}. An easy calculation shows that the latter holds if either $s=3$ and $p\ge 29$ or $s=4,5$ and $p\ge 11$.
\end{proof}

When $s=3,4,5$ and for small values of $p$, our methods for studying inequality \eqref{s>2} are too coarse. We now use a different generic bound for $W(q^n-1)$ that is good for large $q$. 

\begin{proposition}\label{prop:table}
Let $q$ be a power of a prime $p$, $n=ps$, where $s=3, 4, 5$ and $\gcd(p, s)=1$. Additionally, suppose that $p\le 23$ if $s=3$ and $p\le 7$ if $s=4,5$. There exists a constant $C_{p, s}$ such that for any $q=p^t$ with $t>C_{p, s}$, the pair $(q, s)$ satisfies Inequality \eqref{s>2}. Moreover, given $p$ and $s$, their corresponding explicit constant $C_{p, s}$ which is an upper bound for which Inequality~\eqref{s>2} may not hold is given in the following table.
\begin{center}
 \begin{tabular}{|c|c| c| c| c| c| c| c| c|c|} 
 \hline
  $p$ & $2$  & $3$ & $5$ & $7$ & $11$ & $13$ & $17$ &  $19$ &$23$\\ [0.5ex] 
 \hline \hline
 $s=3$& $162$ & -- & $28$ & $16$& $8$ & $6$ & $4$ & $4$& $3$ \\ 
 \hline
 $s=4$ & -- & $8$ & $3$ & $2$& -- & -- & -- & --& --\\ 
 \hline
 $s=5$ &  $10$& $4 $& -- & $1$& -- & -- & -- & --& -- \\ 
 \hline
\end{tabular}
\end{center}
\end{proposition}
\begin{proof}
From Lemma \ref{divisor}, we obtain $d(q^{ps}-1)<q^{\frac{1.066ps}{\log\log (q^{ps}-1)}}$ and clearly we have $W(q^{ps}-1)\le d(q^{ps}-1)$. Recall that $W(T)\le 2^{s-1}\le \frac{q^{s\log_q 2}}{2}$ and then
$$W(T)W(q^{ps}-1)<q^{\frac{1.066ps}{\log\log (q^{ps}-1)}+s\log_q2}.$$
Hence, if $$s \log_q 2+ \frac{1.066ps}{\log\log (q^{ps}-1)}< p\left(\frac{s}{2}-1 \right),$$ 
Inequality \eqref{s>2} holds. But, writing $q=p^t$, the last inequality is equivalent to
\begin{equation}\label{C_ps}\frac{\log 2}{pt\log p}+\frac{1.066}{\log\log (p^{pts}-1)}< \frac{1}{2}-\frac{1}{s}.\end{equation}

It is not hard to see that the left side of Inequality \eqref{C_ps} is a decreasing function in $t$. In particular, for fixed $p$ and $s$, if $t_0=C_{p, s}$ is the greatest positive integer value of $t$ such that the Inequality \eqref{C_ps} is false, then for any $q=p^t$ with $t> C_{p, s}$, the pair $(q, s)$ satisfies Inequality \eqref{s>2}. By a direct calculation, we find the integers $C_{p, s}$.
\end{proof}

\begin{remark}
For the values of $(q = p^t,s)$ which are not covered by Inequality~\eqref{C_ps}, we use Cunningham tables~\cite{Cunningham} to obtain the factorization of $q^{ps}-1$ and find the following genuine exceptions to Inequality~\eqref{s>2}. 
\begin{center}
\begin{tabular}{c||c|c|c}
& $p=2$ & $p=3$ & $p=5$ \\
\hline
$s=3$ & $t=1,2,3,4,5,6,7,8,10,12$ & & $t=1,2$ \\
$s=4$ & & $t=1,2$ & \\
$s=5$ & $t=1,2,3$ & $t=1$ & 
\end{tabular}
\end{center}
For these remaining cases, we find a primitive $1$-normal element by direct search; see Appendix~\ref{sec:pseudocode}.

\end{remark}

\begin{corollary}\label{case2<s<6}
Suppose that $q$ is a power of a prime $p$, let $n=ps$ with $s=3,4,5$ and $\gcd(p, s)=1$, then there exists a primitive $1$-normal element in $\fqn$ over $\fq$. 
\end{corollary} 

\subsubsection{Cases $n=ps, s\ge 6$.}

\begin{proposition}\label{caseS>5}
Suppose that $q$ is a power of a prime $p$ and $n=ps$, where $\gcd(p, s)=1$ and $s\ge 6$. With the exception of the case $(q, s)=(2, 15)$, the pair $(q, s)$ satisfies Inequality \eqref{s>2}.
\end{proposition}

\begin{proof}
Suppose, by contradiction, that $(q, s)$ does not satisfy Inequality \eqref{s>2}. Hence 
\begin{equation}\label{EQM1}W(T)\cdot W(q^{ps}-1)> q^{p(\frac{s}{2}-1)}.
\end{equation}

We divide the proof in cases depending on the value of $q$. 

\paragraph{Case 1.: $q$ odd}
Since $2+\log_2 p>3$ for any odd prime $p$, from items 1 and 2 of Proposition \ref{mainestimation}, 
$$W\left(\frac{q^{ps}-1}{q^s-1}\right)\le q^{\frac{(p-1)s}{3}},$$
for any $s\ge 6$ such that $\mathrm{gcd}(s, p)=1$.
Excluding the pairs in Theorem~\ref{cohenpairs}, for $s\ge 6$ and $q\ge 3$
$$W(T)W(q^s-1)<q^{\frac{s}{2}-1}.$$
Since $W(q^{ps}-1)\le W(q^s-1)W\left(\frac{q^{ps}-1}{q^s-1}\right)$, we conclude that
$$W(T)W(q^{ps}-1)< q^{\frac{(p-1)s}{3}+\frac{s}{2}-1}.$$
Therefore, from Inequality~\eqref{EQM1} we obtain
$$p\left(\frac{s}{2}-1\right)<\frac{s(p-1)}{3}+\frac{s}{2}-1,$$
hence $s<6$, a contradiction with $s\ge 6$.

For the $19$ pairs in Theorem~\ref{cohenpairs} satisfying $q\ge 3$ odd and $s\ge 6$, we directly verify Inequality~\eqref{s>2} either by direct computation or using the bound $$W(q^{ps}-1)<q^{\frac{ps}{\log(\log(q^{ps}-1))}}.$$

\paragraph{Case 2. $q$ even, $q\ge 8$}
Since $s$ is odd, we have $s\ge 7$. First, suppose that either $s\ge 11$ or $q\ge 32$. Therefore, according to Lemma \ref{even}, we have 
$$W(q^s+1)<q^{0.352(s+0.05)}.$$
Under the assumption $s\ge 11$ or $q\ge 32$ even, from Theorem \ref{cohenpairs}, we know that
$$W(T)W(q^s-1)<q^{\frac{s}{2}-1}.$$
Since $W(q^{2s}-1)= W(q^s-1)W(q^s+1)$, combining these inequalities and Inequality \eqref{EQM1} we obtain
$$\frac{s}{2}-1+0.352(s+0.05)>s-2,$$
and then $s<7$, a contradiction.

For the remaining cases $(q, n)=(8, 7), (8, 9), (16, 7)$ and $(16, 9)$ we verify directly that Inequality~\eqref{s>2} is satisfied. 

\paragraph{Case 3. $q=4$} 
First, suppose that $s\ge 15$ is odd. Hence $2^s+1\le 2^{s+0.05}$ and ${\frac{1.1}{\log\log(2^s+1)}}<0.428$. Therefore, from Lemma \ref{divisor} we have 
$$W(2^s+1)\le d(2^s+1)< (2^s+1)^{\frac{1.1}{\log\log(2^s+1)}}<2^{0.428(s+0.05)}=4^{0.214(s+0.05)}.$$ 
According to Lemmas \ref{aux2}, \ref{Cohen} and \ref{even}, for $s\ge 19$ odd we have $W(T)\le 4^{\frac{s+9}{6}}$, $W(2^s-1)<4^{\frac{s}{14}+1}$ and $W(4^s+1)<4^{\frac{s}{4.05}}$, respectively.

Since $W(2^{2s}-1)=W(2^s+1)W(2^s-1)$, we combine all the previous inequalities and obtain
$$W(T)W(4^{2s}-1)< 4^{0.6991s+2.5107},$$
for $s\ge 19$. Therefore, from Inequality \eqref{EQM1} we obtain
$$0.6991s+2.5107>s-2$$
hence $s\le 14$, a contradiction with $s\ge 15$.  For the remaining cases $s=7, 9, 11, 13$ we  verify directly that Inequality \eqref{s>2} is satisfied. 

\paragraph{Case 4. $q=2$} 
First, suppose that $s\ge 24$. Hence $2^s+1<2^{s+1}$ and ${\frac{1.1}{\log\log(2^s+1)}}<0.392$ and then, according to Lemma \ref{divisor} 
$$W(2^s+1)\le d(2^s+1)< (2^s+1)^{\frac{1.1}{\log\log(2^s+1)}}<2^{0.392(s+1)}.$$ 
Also, from Lemma \ref{Cohen} we have $W(2^s-1)<2^{\frac{s}{7}+2}$ and, from Lemma \ref{aux}, we have $W(T)\le 2^{\frac{s+9}{5}}$. Combining these inequalities with Inequality \eqref{EQM1} we obtain:
$$\frac{12s+133}{35}+0.392(s+1)>s-2,$$
hence $s<24$, a contradiction. For the remaining cases $7\le s\le 23$ odd we verify directly that, with the exception of the case $s=15$, Inequality \eqref{s>2} is satisfied. 
\end{proof}

For the excpetional case $q=2$ and $s=15$, we directly find a primitive $1$-normal element of $\F_{q^n}=\F_{q^{30}}$ using the program in Appendix~\ref{sec:pseudocode}. In particular, from Corollaries~\ref{case2<s<6} and \ref{caseS>5}, we obtain the following.

\begin{theorem}
Suppose that $n=ps$, where $s\ge 3$ and $\gcd(p, s)=1$. Then there exists a primitive $1$-normal element of $\F_{q^n}$.
\end{theorem}

\subsection{Cases $n=p$ and $n=2p$}
Here we study the existence of primitive $1$-normals in the case $n=p$ or $n=2p$, where $p$ is odd; we combine estimations that we have used in the previous cases with some combinatorial arguments. We summarize the results in the following.

\begin{proposition}\label{thm:n=sp-s-small}
\begin{enumerate}
\item 
For $n=p$, if there is no primitive $1$-normal element of $\fqn$ over $\fq$, then 
\begin{equation}\label{n=p}\frac{q^{p-2}}{\theta(q^p-1)}\ge q^{p-1}-q^{p/2}W(q^p-1),\end{equation}
which is violated except for $q=p=5$.
\item For $n=2p$, if there is no primitive $1$-normal element of $\fqn$ over $\fq$, then 
\begin{equation}\label{n=2p}
\frac{q^{2p-1}}{(q-1)\cdot \theta(q^{2p}-1)}\ge q^{2p-1}-2q^{p}W(q^{2p}-1),
\end{equation}
which is violated except for $q=p=3$.
\end{enumerate}
\end{proposition}
\begin{proof}
We split the proof into cases
\paragraph{Case $n=p$} 
For the case $p=n=3$, we have already shown the existence of primitive $1$-normals in degree $3$ extensions in Corollary~\ref{cor:small-n}. Suppose now that $p\ge 5$. 

For any $\beta\in \F_{q^p}$ such that $\Tr_{q^p/q}(\beta)=0$, $m_{\sigma, \beta}(x)=(x-1)^d$ for some $0\le d\le p-1$. 
Moreover, an element $\beta\in \F_{q^p}$ is $1$-normal over $\F_q$ if and only if $m_{\sigma, \beta}(x)=(x-1)^{p-1}$. Hence, if there is no primitive $1$-normal element, any primitive element $\beta \in \F_{q^p}$ of trace zero satisfies $(x-1)^{p-2}\circ \beta=0$. Therefore, if $N$ denotes the number of primitive elements of trace zero in $\F_q$, we would have $N\le q^{p-2}$. But, according to Equation \eqref{eqn:main1} of Proposition \ref{propmain} with $n=p, m=1$ and $f(x)=1$, we have Inequality~\eqref{n=p}: 
\begin{equation*}
\frac{q^{p-2}}{\theta(q^p-1)}\ge \frac{N}{\theta(q^p-1)}\ge q^{p-1}-q^{p/2}W(q^p-1).\end{equation*}

Since $q$ is a power of $p$, $\frac{q^p-1}{q-1}$ and $q-1$ are coprime, hence $W(q^p-1)=W\left(\frac{q^p-1}{q-1}\right)W(q-1)$. We have the trivial bound $W(q-1)\le d(q-1)< 2q^{1/2}$ and, according to Proposition \ref{mainestimation}, for $p\ge 5$ (with the exception of $q=5$) the following holds
$$W\left(\frac{q^p-1}{q-1}\right)\le q^{\frac{p-1}{2+\log_2p}}\le q^{\frac{p-1}{2+\log_2 5}}<q^{\frac{p-1}{4.3}},$$
hence  $W(q^p-1)\le 2\cdot q^{\frac{2p+2.3}{8.6}}$. 
Therefore, from Inequalities \eqref{n=p} and \eqref{eqn:ramanujan} we have the following:
$$3.6\log q \cdot q^{p-2}\ge \frac{N}{\theta(q^p-1)}> q^{p-1}-q^{p/2}W(q^p-1)\ge q^{p-1}-2q^{\frac{6.3p+2.3}{8.6}},$$
hence
\begin{equation}\label{E1}3.6\log q>q-2q^{\frac{19.5-2.3p}{8.6}}. \end{equation}

If $p=5$, we get $$q-2q^{\frac{4}{4.3}}<3.6\log q,$$ which is true only if $q=5^t$ with $t\le 6$.

If $p=7$, we get $$q-2q^{\frac{1.7}{4.3}}<3.6\log q,$$ which is true only if $q=7$.

If $p\ge 11$, we have $q-2q^{\frac{19.5-2.3p}{8.6}}> q-2$ and from Inequality \eqref{E1} we conclude that $$q-2<3.6\log q,$$ which is true only for $q<11$, a contradiction with $q\ge p\ge 11$.

For the remaining cases $q=5^t, t=1, \cdots, 6$ and $p=q=7$ we go back to Inequality \eqref{n=p}; by a direct verification we see that, with the exception of the case $q=5$ and $t=1$, Inequality  \eqref{n=p} does not hold. 

\paragraph{Case: $n=2p$} 
Notice that any element $\alpha\in \F_{q^{2p}}$ which is $(x+1)$-free over $\F_q$ and has trace zero satisfies $m_{\sigma, \alpha}(x)=(x-1)^{d}(x+1)^p$ for some $0\le d\le p-1$. Also, notice an $(x+1)$-free element is $1$-normal if and only if $m_{\sigma, \alpha}(x)=(x-1)^{p-1}(x+1)^p$. Hence, if there is no primitive $1$-normal element, any primitive element $\beta \in \F_{q^{2p}}$ of trace-$0$ and $(x+1)$-free satisfies $(x+1)^p(x-1)^{p-2}\circ \beta=0$. Therefore, if $N$ denotes the number of primitive elements of trace zero and $(x+1)$-free in $\F_q$, we have $N\le q^{2p-2}$. But, according to Equation~\eqref{eqn:main1} of Proposition \ref{propmain} with $n=2p, m=1$ and $f(x)=x+1$, we have Inequality~\eqref{n=2p}:
\begin{equation*}
\frac{q^{2p-1}}{(q-1)\cdot \theta(q^{2p}-1)}\ge \frac{N}{\theta(q^{2p}-1)\Theta(x+1)}\ge q^{2p-1}-2q^{p}W(q^{2p}-1).
\end{equation*}

Since $q$ is a power of $p$, $\frac{q^{2p}-1}{q^2-1}$ and $q^2-1$ are coprime, hence $W(q^{2p}-1)=W\left(\frac{q^{2p}-1}{q^2-1}\right)W(q^2-1)$. Also, since $q$ is odd, $q^2\equiv 1\pmod 8$ and we have the trivial bound $$W(q^2-1)\le 2W\left(\frac{q^2-1}{8}\right)\le 2\cdot d\left(\frac{q^2-1}{8}\right)\le \sqrt{2}q.$$ 

We first consider the case $p=3$: according to Proposition \ref{mainestimation}, for $q\ge 3^5$, the following holds
$$W\left(\frac{q^{6}-1}{q^2-1}\right)< q^{0.92},$$
hence $W(q^6-1)<\sqrt{2}q^{1.92}$. According to Inequality \eqref{eqn:ramanujan}, we have
$$\frac{1}{\theta(q^{6}-1)}<3.6\log q+1.25$$
Therefore, from Inequality \eqref{n=2p}, we obtain
$$(3.6\log q+1.25)\frac{q^5}{q-1}>q^{5}-2\sqrt{2}q^{4.92},$$
hence
$$(3.6\log q+1.25)\frac{q}{q-1}>q-2\sqrt{2}q^{0.92},$$
and this is only true for $q=3^t$ with $t\le 11$.

For $p\ge 5$, according to Proposition \ref{mainestimation}, the following holds
$$W\left(\frac{q^{2p}-1}{q^2-1}\right)\le q^{\frac{2(p-1)}{2+\log_2p}}\le q^{\frac{2(p-1)}{2+\log_2 5}}<q^{\frac{2(p-1)}{4.3}},$$
hence $W(q^{2p}-1)<\sqrt{2}q^{\frac{2p+2.3}{4.3}}$.

According to Inequality \eqref{eqn:ramanujan} with $p \geq 5$, we have
$$\frac{1}{\theta(q^{2p}-1)}<3.6\log q+1.25.$$
Therefore, from Inequality \eqref{n=2p} we obtain
$$(3.6\log q+1.25)\frac{q^{2p-1}}{q-1}\ge q^{2p-1}-2\sqrt{2}q^{\frac{6.3p+2.3}{4.3}},$$
hence
$$(3.6\log q+1.25)\frac{q}{q-1}\ge q-2\sqrt{2}q^{\frac{10.9-2.3p}{4.3}}\ge q-2\sqrt{2}q^{\frac{-0.6}{4.3}}.$$
A simple verification shows that this is only true for $q\le 13$. 

For the remaining cases $q=3^t$ for $t\le 11$ and $q=5, 7, 11, 13$ we go back to Inequality \eqref{n=2p}; by a direct verification we see that, with the exception of the case $q=3$ and $t=1$, Inequality \eqref{n=2p} does not hold. 
\end{proof}

Proposition~\ref{thm:n=sp-s-small} shows that there exists a primitive $1$-normal element of $\fqn$ over $\fq$ when $n=p$ or $n=2p$, $p$ odd, and excludes the cases
\begin{enumerate}
\item $p=q=5$, $n=p$;
\item $q=p=3$, $n=2p$.
\end{enumerate}
For these two cases, we directly find a primitive $1$-normal element of $\fqn$ over $\fq$ by the program in Appendix~\ref{sec:pseudocode}.  

\begin{corollary}\label{cor:n=sp-s-small}
Suppose that $q$ is a power of an odd prime $p$. For $n=p$ or $n=2p$, there exists a primitive $1$-normal element of $\F_{q^n}$ over $\F_q$. 
\end{corollary}

We further recall from Lemma~\ref{lem:nonexistn-1} that there cannot exist a primitive $1$-normal for $n=p=2$, and the case $n=2p=4=p^2$ is covered by Proposition~\ref{prop:n=p^2s}.

\section{Conclusion}\label{sec:conclusions}
We recall the main result of this paper. 

\setcounter{section}{1}
\setcounter{theorem}{2}
\begin{theorem} \textup{\textbf{(The Primitive $1$-Normal Theorem)}}
Let $q$ be a prime power and let $n \geq 3$ be a positive integer. Then there exists a primitive $1$-normal element of $\fqn$ over $\fq$. Furthermore, when $n=2$ there is no primitive $1$-normal element of $\F_{q^2}$ over $\fq$. 
\end{theorem}

Theorem~\ref{MainTheorem} provides a complete answer to~\cite[Problem 6.2]{HMPT}. The methods here follow along a similar path to~\cite{HMPT}, but rely on a variety of technical reductions and further estimations provided in Appendix~\ref{sec:estimations}. When the estimations are not powerful enough to give a complete solution, we turn to direct calculations and computer searches using the \emph{Sage} computer algebra system~\cite{sagemath}; see Appendix~\ref{sec:pseudocode}.

When $n-2 \geq k\geq 1$, $k$-normal elements of $\fqn$ over $\fq$ may not exist. Determining the values of $k$ for which $k$-normal elements of $\fqn$ over $\fq$ do exist requires knowledge of the irreducible factorization of $x^n-1$. On one extreme, when $n$ is prime and $q$ is primitive modulo $n$, then $k=0,1,n-1,n$ are the only values for which there exist $k$-normal elements of $\fqn$ over $\fq$, and in this case primitive $k$-normals exist if and only if $k=0,1$. On the other extreme, if $n$ is a power of the characteristic $p = \mathrm{char}(\fq)$, it can be shown (for example, using primary decompositions) that $k$-normal elements exist for all $0 \leq k \leq n$, and it is unknown whether primitive $k$-normals exist for $2 \leq k \leq n-2$. Problem 6.3 of~\cite{HMPT} asks for a complete solution for which $q,k,n$ there exists primitive $k$-normal elements of $\fqn$ over $\fq$. 

We remark that $k$-normality is an additive analogy to multiplicative order of a finite field element. A similar analysis as the in this paper could be performed for when an element is simultaneously $k$-normal (say, for $k=0,1$) and has ``high-order''; for example, if its order is $(q^n-1)/p_1$, where $p_1$ is the largest prime dividing $q^n-1$. This is Problem 6.4 of~\cite{HMPT}. Of course, it is possible (for Mersenne primes) that $2^n-1$ is itself prime, hence yields no prime factors other than itself. 

Recognizing restrictions such as noted above, determining the existence of elements of various $k$-normality and orders may provide interesting analytic work, but these questions are somewhat less natural as they do not apply universally for all finite fields.

\begin{center}{\bf Acknowledgments}\end{center}
This work was conducted during a scholarship of the first author, supported by the Program CAPES-PDSE (process - 88881.134747/2016-01), at Carleton University.


\appendix

\section{Appendix: Estimations for $W(q^n-1)$ and $W\left(\frac{x^n-1}{x-1}\right)$}\label{sec:estimations}

In this Appendix, we present a number of general estimations which we will use to bound the quantities $W(q^n-1)$ and $W(f)$ appearing in Equation~\eqref{eqn:main1}. When no number theoretic argument can be effective to bound the number $W(q^n-1)$, we just use the bound $W(q^n-1)\le d(q^n-1)$, where $d(m)$ denotes the number of divisors of $m$; frequently we use the bound $d(m)<2\sqrt{m}$ and, as we will see, we have a better estimation for $m$ large.

We start with some estimations in even characteristic.

\begin{lemma}\label{aux}
Suppose that $n\ge 3$ is odd and $T(x):=\frac{x^n-1}{x-1}\in \F_2[x]$. Then $W(T)\le 2^{\frac{n+9}{5}}$.
\end{lemma}

\begin{proof}
For each $2\le i\le 4$, let $s_i$ be the number of irreducible factors of degree $i$ dividing $T(x)$. Since $n$ is odd, $T(x)$ has no linear factor. By a direct verification we see that the number of irreducible polynomials over $\F_2$ of degrees $2, 3$ and $4$ are $1, 2$ and $3$, respectively. Hence $s_2\le 1, s_3\le 2$ and $s_4\le 3$. In particular, the number of irreducible factors of $T(x)$ over $\F_2$ is at most $$\frac{n-1-2s_2-3s_3-4s_4}{5}+s_2+s_3+s_4=\frac{n-1+3s_2+2s_3+s_4}{5}.$$Since $\frac{n-1+3s_2+2s_3+s_4}{5}\le \frac{n-1+3+4+3}{5}=\frac{n+9}{5}$, we conclude the proof.
\end{proof}

\begin{lemma}\label{aux2}
Suppose that $n\ge 3$ is odd and $T(x):=\frac{x^n-1}{x-1}\in \F_4[x]$. Then $W(T)\le 2^{\frac{n+9}{3}}$.
\end{lemma}

\begin{proof}
For $i=1, 2$, let $s_i$ be the number of irreducible factors of degree $i$ dividing $T(x)$. Since $n$ is odd, $T(x)$ has no repeated irreducible factors and it is not divisible by $x-1$. Clearly $T(x)$ is not divisible by $x$. By a direct verification we see that the number of irreducible polynomials over $\F_4$ of degrees $1$ (different from $x, x-1$) and $2$ are $2$ and $6$, respectively. Hence $s_1\le 2, s_2\le 6$. In particular, the number of irreducible factors of $T(x)$ over $\F_4$ is at most $$\frac{n-1-s_1-2s_2}{3}+s_1+s_2=\frac{n-1+2s_1+s_2}{3}.$$ Since $\frac{n-1+2s_1+s_2}{3}\le \frac{n-1+4+6}{3}=\frac{n+9}{3}$, we conclude the proof.
\end{proof}

\begin{lemma}\textup{\cite[Lemma 7.5]{cohen3}}\label{Cohen} Suppose that $n$ is an odd number. Then $$W(2^n-1)< 2^{\frac{n}{7}+2}.$$\end{lemma}

We now introduce some results in arbitrary characteristic. 

\begin{lemma}\label{lem:ramanujan}
Suppose that $q$ is a power of a prime $p$, where $p\ge 5$ or $p=3$ and $q\ge 27$. For $s\ge 1$, we have
\begin{equation}\label{eqn:ramanujan}\theta(q^{ps}-1)^{-1}=\frac{q^{ps}-1}{\varphi(q^{ps}-1)}< 3.6 \log q+1.8\log s.\end{equation}
\end{lemma}

\begin{proof}
It is well known that $$\frac{n}{\varphi(n)}\le e^{\gamma}\log\log n+\frac{3}{\log\log n},$$ for all $n\ge 3$, where $\gamma$ is the Euler constant and $1.7<e^{\gamma}<1.8$. Also, since $e^x\ge 1+x$ for any $x\ge 0$, we have $\log \log q\le \log q -1$. Hence $\log\log (q^{ps}-1)<\log s+\log p+\log\log q \le \log s+ 2\log q-1$.

By the hypothesis, $\log \log (q^{ps}-1)\ge \min\{\log\log (5^5-1), \log\log (27^3-1)\} >2$, hence $\frac{3}{\log\log (q^{ps}-1)}\le 1.5$ and so we get the following:
\[\theta(q^{ps}-1)^{-1}\le 1.8(\log s+ 2\log q-1)+1.5< 3.6 \log q+1.8\log s. \qedhere \]
\end{proof}

\begin{lemma}\label{numthe}
Let $q$ be a power of an odd prime $p$ and $s$ a positive integer such that $\gcd(p, s)=1$. Then any prime divisor $r$ of $\frac{q^{ps}-1}{q^s-1}$ satisfies $r\equiv1 {\pmod {2p}}$.
\end{lemma}
\begin{proof}
First, notice that $$\gcd\left (\frac{q^{ps}-1}{q^s-1}, q^s-1\right)=\gcd(p, q^s-1)=1.$$ Let $r$ be any prime divisor of $\frac{q^{ps}-1}{q^s-1}$. In particular, $r$ divides $q^{ps}-1$ and does not divide $q^s-1$. Hence, if $l=\ord_r q$ we know that $l$ divides $ps$ but does not divide $s$ and then it follows that $l$ is divisible by $p$. We know that $l$ divides $\varphi(r)=r-1$ and then $p$ divides $r-1$. Therefore $r-1=dp$ for some positive integer $d$. But, since $p\ge 3$ is odd, we have $r>2$ and then $r$ is an odd prime number. Thus $d$ is even and then $r\equiv 1\pmod {2p}$.
\end{proof}

In a similar proof of the previous Lemma, we also obtain the following.

\begin{lemma}\label{five}
Let $q$ be a power of a prime $p$. Then any prime $r$ dividing $\frac{q^5-1}{q-1}$ such that $r$ does not divide $q-1$, satisfies $r\equiv 1\pmod {10}$. 
\end{lemma}

Lemmas \ref{numthe} and \ref{five} show that the primes dividing the numbers $\frac{q^a-1}{q-1}$ (for specific values of $a$) are not too small. In particular, we can obtain effective bounds for expressions of the form $W\left(\frac{q^a-1}{q-1}\right)$. We exemplify this in the following two results.

\begin{lemma}\label{estimation5}
Let $q$ be a power of a prime such that $q\ge 2^{17}$. Then 
$$16\cdot W(q^5-1)<q^{1.5}.$$
\end{lemma}

\begin{proof}
Let $q_0$ be the greatest divisor of $\frac{q^5-1}{q-1}$ that is relatively prime to $q-1$. Clearly $W(q^5-1)\le W(q_0)W(q-1)$ and $q_0\le \frac{q^5-1}{q-1}=q^4+q^3+q^2+q+1<1.01q^4$ for $q\ge 2^{17}$.

Let $r_1<r_2<\cdots<r_d$ be the list of distinct primes dividing $q_0$. From Lemma \ref{five} we know that $r_i\ge 10i+1$, hence 
$$q_0\ge r_1\cdots r_d\ge 10^d\cdot d!+1>10^d\cdot d!.$$

We first suppose that $d\ge 13$. It follows from induction that, for $d\ge 13$, the inequality $d!\ge 1.01\cdot 2^{2.5d}$ holds and we get 
$$q^4>\frac{q_0}{1.01}>\frac{10^d\cdot d!}{1.01}\ge (10\cdot 2^{2.5})^d>2^{5.82d},$$
hence $W(q_0)=2^d<q^{\frac{4}{5.82}}<q^{0.69}$. We have the trivial bound $W(q-1)\le 2\sqrt{q}$ and then
$$16\cdot W(q^5-1)\le 16\cdot W(q_0)W(q-1)<32\cdot q^{1.19}<q^{1.5},$$   
since $q\ge 2^{17}$. For $d\le 12$, we have $W(q_0)\le 2^{12}$ and, since $W(q-1)\le 2\sqrt{q}$, we get
$$16\cdot W(q^5-1)\le 16\cdot W(q_0)W(q-1)\le 2^{17}\cdot \sqrt q\le q^{1.5},$$
since $q\ge 2^{17}$.
\end{proof}

We use some similar ideas from the proof of the previous Lemma and obtain general effective bounds for quantities $W\left(\frac{q^a-1}{q-1}\right)$.

\begin{proposition}\label{mainestimation}
Let $q$ be a power of an odd prime $p$ and $s$ be a positive integer such that $\gcd(p, s)=1$. 
\begin{enumerate}
\item If $p\ge 5$, with the exception of the case $(p, q, s)=(5, 5, 1)$, the following holds $$W\left(\frac{q^{ps}-1}{q^s-1}\right)\le q^{\frac{(p-1)s}{2+\log_2 p}}.$$
\item For $s\ge 6$ and $\gcd(s, 3)=1$,  $$W\left(\frac{3^{3s}-1}{3^s-1}\right)\le 3^{\frac{2s}{3}}.$$
\item For $p=3$ and $q\ge 3^5$,
$$W\left(\frac{q^{6}-1}{q^2-1}\right)\le q^{0.92}.$$
\end{enumerate}
\end{proposition}
\begin{proof}
We analyze the items in the statement by cases. 

\paragraph{Case 1. $q \geq 5$} 
Let $r_1<\cdots<r_d$ be the list of primes dividing $\frac{q^{ps}-1}{q^s-1}$. It follows from Lemma \ref{numthe} that $r_i\ge 2ip+1>2ip$. Hence,
\begin{equation}\label{primes}\frac{q^{ps}-1}{q^s-1}\ge r_1\dots r_d> d!2^dp^d.\end{equation}
We divide in two cases.

\paragraph{Case 1a. $d\le 3$} 
In this case, $W\left(\frac{q^{ps}-1}{q^s-1}\right)\le 8$.

If $p=5$ we have 
$$q^{\frac{(5-1)s}{2+\log_2 5}}>q^{0.92s}>8$$
with the exception of the case $q=5, s=1$.

If $p\ge 7$, note that $f(x)=\frac{x-1}{2+\log_2 x}$ is an increasing function for $x\ge 2$. Since $f(7)>1.24$ we have
 $$q^{\frac{(p-1)s}{2+\log_2 p}}>q^{1.24 s}>8$$
for all $q$ a power of a prime $p\ge 7$.

\paragraph{Case 1b. $d\ge 4$}
In this case, it follows by induction that $d!\ge \frac{3}{2}\cdot 2^d$. From Inequality \eqref{primes} we obtain the following:
$$\frac{q^{ps}-1}{q^s-1}>\frac{3}{2}(4p)^d=\frac{3}{2}\cdot 2^{d(2+\log_2 p)}.$$
To finish the proof, just notice that $W\left(\frac{q^{ps}-1}{q^s-1}\right)=2^d$ and $$\frac{q^{ps}-1}{q^s-1}=q^{(p-1)s}+\frac{q^{(p-1)s}-1}{q^s-1}\le q^{(p-1)s}+2q^{(p-2)s}\le \frac{3}{2}q^{(p-1)s},$$
for $q\ge 5$.

\paragraph{Case 2. $s \geq 6$ and $\gcd(s,3) = 1$} 
Suppose that $W\left(\frac{3^{3s}-1}{3^s-1}\right)=2^d$. Similarly as before, we obtain
\begin{equation}\label{primes1}\frac{3^{3s}-1}{3^s-1}\ge r_1\dots r_d> 6^d\cdot d!.\end{equation}

If $d\le 6$, we have $W\left(\frac{3^{3s}-1}{3^s-1}\right)\le 64< 81\le 3^{\frac{2s}{3}}$ for $s\ge 6$. For $d\ge 7$, it follows by induction that $d!>\left(\frac{10}{3}\right)^d$ and $20^d>2\cdot 2^{3d}$. Hence, from Inequality \eqref{primes1} we obtain
$$2\cdot 2^{3d}<20^d=6^d\cdot \left(\frac{10}{3}\right)^d<6^d\cdot d!<\frac{3^{3s}-1}{3^s-1}.$$
To finish the proof, just notice that $\frac{3^{3s}-1}{3^s-1}< 2\cdot 3^{2s}$.

\paragraph{Case 3. $p = 3$ and $q \geq 3^5$} 
Suppose that $W\left(\frac{q^{6}-1}{q^2-1}\right)=2^d$. If $d\le 7$, we have $W\left(\frac{q^{6}-1}{q^2-1}\right)\le 128< 243^{0.92}\le q^{0.92}$. For $d\ge 8$, it follows by induction that $d!> \frac{100}{81}(3.5)^d$. Also, notice that
$$\frac{q^{6}-1}{q^2-1}=q^4+q^2+1<q^4(1+q^{-2})^2<\frac{100}{81}q^{4},$$
since $q>9$.
Similarly to before, we obtain 
$$6^d\cdot d!<\frac{q^{6}-1}{q^2-1}.$$
By a simple calculation we conclude that, for $d\ge 8$ we have
$$\frac{100}{81}21^d\le \frac{100}{81}q^4,$$ 
hence $2^d\le q^{\frac{4}{\log_2 21}}<q^{0.92}$ and we are done.
\end{proof}

We finally present a general (but not sharp) bound for the number $d(m)$, which is fair for sufficiently large $m$.

\begin{lemma}\label{divisor} If $d(m)$ denotes the number of divisors of $m$, then for all $m\ge 3$,
$$d(m)\le m^{\frac{1.5379 \log 2}{\log\log m}}<m^{\frac{1.066}{\log\log m}}<m^{\frac{1.1}{\log\log m}}.$$
\end{lemma}
\begin{proof} 
This inequality is a direct consequence of the result in \cite{divisor}.
\end{proof}

Combining Lemma \ref{divisor} and some previous ideas, we find general bounds for the numbers $W(q^s+1)$ in the case that $q$ is even.

\begin{lemma}\label{even}
\begin{enumerate}
\item Suppose that $q\ge 8$ is a power of two and $s\ge 7$ is odd. If either $s\ge 11$ or $q\ge 32$ we have
$$W(q^{s}+1)<q^{0.352(s+0.05)}.$$
\item For $s\ge 17$ odd,
$$W(4^s+1)<4^{\frac{s}{4.05}}.$$

\end{enumerate}
\end{lemma}

\begin{proof}
\begin{enumerate}
\item Clearly $W(q^{s}+1)<d(q^{s}+1)$ and, according to Lemma \ref{divisor}, we have $d(q^{s}+1)<(q^s+1)^{\frac{1.1}{\log\log(q^{s}+1)}}$. Note that, since $q\ge 8$ and $s\ge 7$, it follows that $q^{0.05}>1.1 >1+\frac{1}{q^s}$ and then $q^s+1<q^{s+0.05}$. Also, if either $s\ge 11$ or $q\ge 32$, under our conditions we have
$$q^s+1\ge \min\{8^{11}+1, 32^7+1\}=2^{33}+1.$$ 
To finish the proof, just notice that $\frac{1.1}{{\log\log(2^{33}+1)}}<0.352$.

\item Let $r$ be any prime divisor of $4^s+1$. Hence $2^{2s}\equiv -1\pmod r$ and, in particular, $-1$ is a quadratic residue $\pmod r$ and then $r\equiv 1\pmod 4$. Let $p_1<\cdots <p_d$ be the prime divisors of $4^s+1$. In particular, we have proved that $p_i\ge 4i+1$. Hence
$$4^s+1\ge p_1\dots p_d\ge 4^d\cdot d!+1.$$

Notice that if $d\le 8$, then $W(4^s+1)\le 256= 4^{4}<4^{\frac{s}{4.05}}$ since $s\ge 17$. For $d\ge 9$, it follows by induction that $d!>2^{2.05d}$ and then
$$4^s\ge 4^d\cdot d!>2^{4.05\cdot d}.$$
To finish the proof, note that $W(4^s+1)=2^d<4^{\frac{s}{4.05}}$.\qedhere
\end{enumerate}
\end{proof}

\section{Appendix: Pseudocode for search for primitive $1$-normals}\label{sec:pseudocode}


Algorithm~\ref{alg:search} presents a search routine for primitive $1$-normal elements of $\fqn$ over $\fq$. It relies on the original characterization of $k$-normal elements from~\cite{HMPT}. 

\begin{theorem}\label{thm:k-normalgcd}
Let $\alpha \in \fqn$ and let $g_\alpha(x) = \sum_{i=0}^{n-1} \alpha^{q^i}x^{n-1-i} \in \fqn[x]$. Then $\gcd(x^n-1, g_\alpha(x))$ has degree $k$ if and only if $\alpha$ is a $k$-normal element of $\fqn$ over $\fq$.
\end{theorem}

Algorithm~\ref{alg:search} proceeds as follows. Let $\fqn \cong \fq[x]/(f)$ with $f$ a primitive polynomial, and let $g$ be a root of $f$. Hence, $\langle g \rangle = \fqn^*$ and $g^i$ is primitive if and only if $\gcd(i, q^n-1) = 1$. 
For each primitive element, check its $k$-normality using Theorem~\ref{thm:k-normalgcd}. If $k=1$, then the resulting element is primitive $1$-normal and is returned. The algorithm returns ``Fail'' if no primitive $1$-normal is found after $q^n-2$ iterations; that is, if all of $\fqn$ is traversed. 

We implemented Algorithm~\ref{alg:search} with the Sage computer algebra system~\cite{sagemath}.

\begin{algorithm}[h]
\caption{Pseudocode for primitive $1$-normal element search algorithm}
\label{alg:search}
\begin{algorithmic}
\State{{\bf Input:} positive integers $q,n$}
\State {\bf Returns:} primitive $1$-normal element: ${elt} \in \fqn$; otherwise ``Fail''
\State{$\mathrm{mult\_order} \gets q^n-1$}
\State{$g \gets \mathrm{generator}(\fqn^*)$}
\State{${cyclo} \gets x^n-1 \in \fqn[x]$}
\State{}

\Function{$\mathrm{check\_k\_normal}$}{$v$} \Comment{See Theorem~\ref{thm:k-normalgcd}}
   \State $g_v(x) \gets v x^{n-1} + v^q x^{n-2} + \cdots + v^{q^{n-1}}$
	 \State $k \gets \deg(\gcd(g_v, {cyclo}))$
	 \State \Return $k$
\EndFunction

\State{}
\State{$i \gets 1$}
\While{True}
   \If{$i$ {\bf is} $\mathrm{mult\_order}$} \Comment{No primitive $1$-normals found in $\fqn$}
	   \State \Return{``Fail''}
	 \EndIf 
	
	 \If{$\gcd(i,\mathrm{mult\_order}) \neq 1$} \Comment{Only check primitive elements}
	    \State{$i \gets i+1$}
	    \State{{\bf continue}}
	 \EndIf
	 
	 \State{$\mathrm{elt} \gets g^i$}
	 \State{$k \gets \mathrm{check\_k\_normal}({elt})$}
	 \If{$k$ {\bf is} $1$}
	   \State \Return $elt$
	 \EndIf
   \State{$i \gets i+1$}
\EndWhile
\end{algorithmic}
\end{algorithm}

\end{document}